\documentclass[11pt]{amsart}
\usepackage[T1]{fontenc}
\usepackage{a4wide}
\usepackage{mathpazo}
\usepackage{mathrsfs}
\usepackage{amssymb, amsmath}
\usepackage{enumitem}
\usepackage{tikz}

\newcommand{\sC}{\mathscr C}
\newcommand{\NN}{\mathbf{N}}
\newcommand{\ZZ}{\mathbf{Z}}

\newcommand{\SL}{\mathbf{SL}}
\newcommand{\St}{\mathbf{St}}
\newcommand{\gr}{\mathfrak g}

\newcommand{\se}{\subseteq}

\newcommand{\fhi}{\varphi}
\newcommand{\lra}{\longrightarrow}
\newcommand{\inv}{^{-1}}
\newcommand{\Heis}{\mathrm{Heis}}
\newcommand{\Hig}{\mathrm{Hig}}
\theoremstyle{plain}
\newtheorem{thm}{Theorem}
\newtheorem*{thm*}{Theorem}
\newtheorem{lem}[thm]{Lemma}
\newtheorem{prop}[thm]{Proposition}

\theoremstyle{definition}
\newtheorem*{defn*}{Definition}

\newtheorem*{example*}{Example}

\newtheorem*{rem*}{Remark}
\begin{document}

\title[Variations on a theme by Higman]{Variations on a theme by Higman}

\begin{abstract}
We propose elementary and explicit presentations of groups that have no amenable quotients and yet are SQ-universal. Examples include groups with a finite $K(\pi, 1)$, no Kazhdan subgroups and no Haagerup quotients.
\end{abstract}
\author{Nicolas Monod}
\address{EPFL, 1015 Lausanne, Switzerland}
%
%
\date{April 2016}
\maketitle

\section{Introduction}
In 1951, G.~Higman defined the group
\begin{equation}\label{eq:Hig}
\Hig_n = \big\langle a_i \ (i\in \ZZ/n\ZZ) : [a_{i-1}, a_i]= a_i \big\rangle
\end{equation}
and proved that for $n\geq 4$ it is infinite without non-trivial finite quotient~\cite{Higman51}. Since the presentation~\eqref{eq:Hig} is explicit and simple, A.~Thom suggested that $\Hig_n$ is a good candidate to contradict approximation properties for groups and proved such a result in~\cite{Thom_higman}. Perhaps the most elusive approximation property is still \emph{soficity}~\cite{Gromov_ax,Weiss00}; but a non-sofic group would in particular not be residually \emph{amenable}, a statement we do not know for the Higman groups (cf.\ also~\cite{Helfgott-Juschenko_arx}). The purpose of this note is to propound variations of Higman's construction with no non-trivial amenable quotients at all.

\bigskip
There are several known sources of groups without amenable quotients since it suffices to take a (non-amenable) \emph{simple} group to avoid all possible quotients. However, as Thelonius Sphere Monk observed, \emph{simple ain't easy}. To wit, one had to wait until the break-through of Burger--Mozes~\cite{Burger-Mozes0,Burger-Mozes2} for simple groups of \emph{type~F}, i.e.\ admitting a finite $K(\pi, 1)$. Before this, no torsion-free finitely presented simple groups were known.

\bigskip
The examples below are of a completely opposite nature because they admit a wealth of quotients: indeed, like $\Hig_n$, they are \emph{SQ-universal}, i.e.\ contain any countable group in a suitable quotient. It follows that they have uncountably many quotients~\cite[\S{III}]{Neumann37}, despite having no amenable quotients.

We shall start with the easiest examples, whose cyclic structure is directly inspired by~\eqref{eq:Hig}. Below that, we propose a cleaner construction, starting from copies of $\ZZ$ only, which might be a better candidate to contradict approximation properties; the price to pay is to replace the cycle by a more complicated graph.

\bigskip\noindent
\textbf{Disclaimer.}
No claim is made to produce the first examples of groups with a hodgepodge of sundry properties (for instance, if $G$ is a Burger--Mozes group, then $G*G$ satisfies many properties of $G_n$ in Theorem~\ref{thm:G} below, though with ``amenable'' instead of ``Haagerup'').\\ Our goal is to suggest transparent presentations for which the stated properties are explicit and their proofs effective.

\bigskip
\noindent
\textbf{1.A. Starting from large groups.}
Given a group $K$, an element $x\in K$ and a positive integer $n$, we define the group
\begin{equation*}
K^{(n,x)} = \big\langle K_i \ (i\in \ZZ/n\ZZ) : [x_{i-1}, x_i]= x_i \big\rangle,
\end{equation*}
where $K_i, x_i$ denote $n$ independent copies of $K,x$. Thus, $K^{(n,e)}=K^{*n}$ and $\Hig_n = \ZZ^{(n,1)}$.

We recall that a group is \emph{normally generated} by a subset if no proper normal subgroup contains that subset. Following the ideas of Higman and Schupp, we obtain:

\begin{prop}\label{prop:var}
Let $K$ be a group normally generated by an element $x$ of infinite order and let $n\geq 4$.

\begin{enumerate}[label=(\roman*)]
\item If $K$ has no infinite amenable quotient (e.g.\ if $K$ is Kazhdan), then $K^{(n,x)}$ has no non-trivial amenable quotient.\label{pt:var:no}
\item If $K$ is finitely presented, torsion-free, type~F$_\infty$, or type~F, then $K^{(n,x)}$ has the corresponding property.\label{pt:var:fin}
\item Every countable group embeds into some quotient of $K^{(n,x)}$.\label{pt:var:SQ}
\end{enumerate}
\end{prop}

\begin{rem*}
Suppose that $\sC$ is any class of groups closed under taking subgroups. The proof of~\ref{pt:var:no} shows: if every quotient of $K$ in $\sC$ is finite, then $K^{(n,x)}$ has no non-trivial quotient in $\sC$. For instance, if $K$ is Kazhdan, then $K^{(n,x)}$ has no non-trivial  quotient with the Haagerup property.
\end{rem*}

\begin{example*}
  The group $K=\SL_d(\ZZ)$ is an infinite, finitely presented (even type F$_\infty$) Kazhdan group for all $d\geq 3$ and the Steinberg relations show that it is normally generated by any elementary matrix (with coefficient~$1$). Alternatively, the Steinberg group itself $K=\St_d(\ZZ)$ has the same properties (it is Kazhdan because it is a finite extension of $\SL_d(\ZZ)$, see e.g.~\cite[10.1]{MilnorK}). This gives us the following presentations of SQ-universal type F$_\infty$ groups without Haagerup quotients:
$$S_{d,n} = \Bigg\langle  E_i^{p,q}\ (i\in \ZZ/n\ZZ, 1\leq p\neq q \leq d) : \begin{array}[c]{l}
{[E_i^{p,q}, E_i^{q,r}] = E_i^{p,r}}\ (p\neq r \neq q)\\
{[E_i^{p,q}, E_i^{r,s}]=e}\ (q\neq r, p\neq s\neq r)\\
{[E_{i-1}^{1,2}, E_i^{1,2}]=E_i^{1,2}}
\end{array}\Bigg\rangle.$$
The choice of the pair $(1,2)$ is arbitrary and any other elementary matrix for $x$ gives an isomorphic group. If we use the Magnus--Nielsen presentation~\cite{Magnus35,Steinberg85s} of $\SL_d(\ZZ)$ instead of the Steinberg group, we have to add the relations $(E_i^{1,2} (E_i^{2,1})\inv E_i^{1,2})^4=e$.

These groups are not, however, torsion-free. Although congruence subgroups of $\SL_d(\ZZ)$ are torsion-free (and even type~F by~\cite{Raghunathan68}), the latter are never normally generated by a single element because they have large abelianizations.

This construction can be transposed to other Chevalley groups.
\end{example*}

Notice that if in addition $K$ is \emph{just infinite}, like for instance $K=\SL_d(\ZZ)$ for $d$ odd~\cite{Mennicke65}, then this construction shows that $K$ embeds into all non-trivial quotients of $K^{(n,x)}$, such as for instance the simple quotients obtained from maximal normal subgroups.

\bigskip
\noindent
\textbf{1.B. An example built from $\ZZ$.}
Consider the semi-direct product
$$L = (\ZZ[1/2])^2 \rtimes (\ZZ \times F_2)$$
where the generator $h$ of $\ZZ$ acts on $(\ZZ[1/2])^2$ by multiplication by~$2$, and the generators $u,v$ of the free group $F_2$ act by multiplication by
$$\begin{pmatrix}1&1\\ 0&1\end{pmatrix}\ \text{ and } \ \begin{pmatrix}1&0\\ 1&1\end{pmatrix}$$
respectively. In particular the group $L$ is torsion-free, linear and finitely presented. It is generated by $\{x,y,h,u,v\}$ where $(x,y)$ is the standard basis of $\ZZ^2$.

\smallskip
We define a group $G_n$ by fusing together $n$ copies $L_i$ of $L$ in a circular fashion along the corresponding generators as follows:
\begin{equation}\label{eq:G}
G_n=\big\langle L_i : (h_i, u_i, v_i) = (y_{i-2}, x_{i-2}, y_{i-1} ), i\in \ZZ/n\ZZ\big\rangle.
\end{equation}
It is easy to write down an explicit presentation of $G_n$. Observe first that $L$, with our choice of generators, has a presentation with the following set~$R$ of relations
$$R(x,y,h,u,v) :
\begin{array}[t]{l}
{e=[x,y]=[x,u]=[y,v]=[h,u]=[h,v],}\\
{[h,x]=[u,y]=x, \ [h,y]=[v,x]=y.}
\end{array}$$
Now~\eqref{eq:G} is equivalent to the finite presentation
\begin{equation}\label{eq:Gn}
G_n=\big\langle x_i, y_i : R(x_i, y_i, y_{i-2}, x_{i-2}, y_{i-1}), i\in \ZZ/n\ZZ\big\rangle.
\end{equation}
We find these groups more elementary than~$K^{(n,x)}$ (with Kazhdan $K$) and hope that they will be easier to use in applications. In return, we have to work more than before to deduce some of the following properties.

\begin{thm}\label{thm:G}
Let $n\geq 8$.

\begin{enumerate}[label=(\roman*)]
\item The group $G_n$ has no non-trivial Haagerup quotient.\label{pt:G:no}
\item Any quotient with a $\frac{1}{36}$-F{\o}lner set for the generators $x_i, y_i$ is trivial.\label{pt:G:Folner}
\item The only Kazhdan subgroup of $G_n$ is the trivial group.\label{pt:G:noT}
\item The group $G_n$ admits a finite $K(\pi, 1)$.\label{pt:G:fin}
\item The group $G_n$ can be constructed starting from copies of $\ZZ$, using amalgamated free products, semi-direct products and HNN-extensions.\label{pt:G:Z}
\item Every countable group embeds into some quotient of $G_n$ if $n\geq 9$.\label{pt:G:SQ}
\item The groups $G_m$ are trivial for $m\leq 4$ and $m=6$.\label{pt:G:trivial}
\end{enumerate}
\end{thm}

\noindent The restriction $n\geq 9$ is probably not needed in~\ref{pt:G:SQ} but makes it very easy to check Schupp's criterion for SQ-universality. We have not elucidated $G_5$ and $G_7$.

\bigskip

\noindent
\textbf{Scholium.}
We should like to point out a general type of presentations subsuming the examples above. Consider a group $L$ and two finite sets $A, P\se L$. We think of elements in $A$ as ``active'', whilst those in $P$ are ``passive''. Consider furthermore a transitive labelled oriented graph $\gr$ whose edges are labelled by $P\times A$. To every vertex $i$ of $\gr$ we associate an independent copy $L_i$ of $L$. We then form the group
$$G =\big\langle L_i, i\in \gr : p_j = a_k \text{ if $\exists$ $(p,a)$-labelled edge from $j$ to $k$} \big\rangle.$$
In order to get a manageable group from this presentation, we would like to ensure at the very least that each $L_i$ embeds. A favourable case is when $A$ is a basis for a free subgroup in $L$ and the edges spread the passive elements of $P_j$ incoming to a vertex $k$ over copies $L_j$ for suitably distinct $j$s. (In our case, we allowed a commutation in $A_k$ because it was going to hold also among the corresponding $P_j$.)

The trade-off is that this spreading  should remain limited compared to the girth of the cycles in $\gr$ along which we can cut the amalgamation scheme. Higman's groups and the groups $K^{(n,x)}$ use a simple $n$-cycle for~$\gr$; as for $G_n$, we depict its graph in the figure below for $n=8$; the orientation is implicit from the labelling.

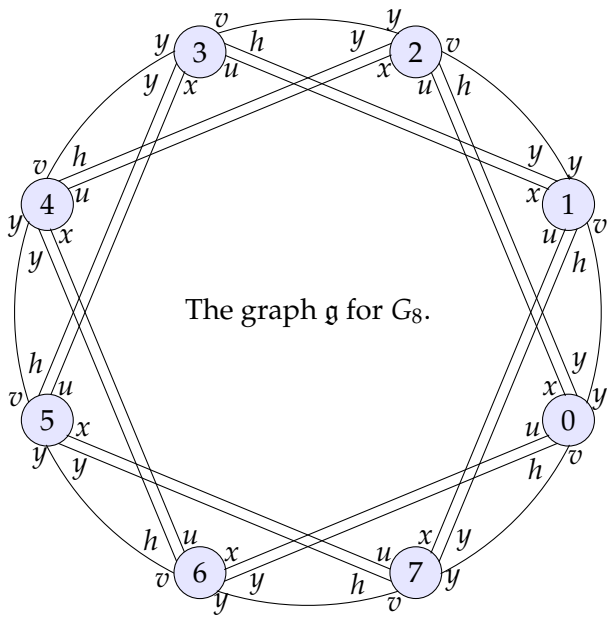
\begin{figure}[h!]
%
%
%

\begin{tikzpicture}[scale=0.75]
\def \num {8}
\def \n {7}  
\def \radius {5cm}
\def \oradius {5.2cm} 
\def \oyradius {5.4cm} 
\def \iradius {4.8cm} 
\def \ixradius {4.5cm} 
\def \iyradius {4.9cm} 
\def \margin {5.2}
\def \omargin {4.4}
\def \oymargin {6}
\def \imargin {4.5}
\def \ixmargin {5}
\def \iymargin {12}
\def \start {-0.5} 

\foreach \s in {0,...,\n }
{
\node[draw, circle, fill=blue!10] at ({360/\num * (\s + \start)}:\radius) {$\s$};

\draw ({360/\num * (\s + \start)+\omargin}:\oradius) arc ({360/\num * (\s +\start)+\omargin}:{360/\num * (\s+1+\start)-\omargin}:\oradius);

\draw ({360/\num * (\s + \start)+\margin}:\radius) -- ({360/\num * (\s + 2 + \start)-\margin}:\radius);
\draw ({360/\num * (\s + \start)+\imargin}:\iradius) -- ({360/\num * (\s + 2 + \start)-\imargin}:\iradius);

\node at ({360/\num * (\s + \start)+\oymargin}:\oyradius) {$y$};
\node at ({360/\num * (\s + \start)+\iymargin}:\iyradius) {$y$};
\node at ({360/\num * (\s + \start)+\ixmargin}:\ixradius) {$x$};

\node at ({360/\num * (\s + \start)-\oymargin}:\oyradius) {$v$};
\node at ({360/\num * (\s + \start)-\iymargin}:\iyradius) {$h$};
\node at ({360/\num * (\s + \start)-\ixmargin}:\ixradius) {$u$};
}

\node at (0,0) {The graph $\gr$ for $G_8$.};

\end{tikzpicture}
\end{figure}

\bigskip\noindent
\textbf{Notation.}
Our convention for commutators is $[\alpha, \beta] = \alpha \beta \alpha\inv \beta\inv$; Higman used a different convention for~\eqref{eq:Hig} but this does not affect the group~$\Hig_n$. Given a subset $E$ of a group $H$, we denote the subgroup it generates by $\langle E \rangle$ or by $\langle E \rangle_H$ when $H$ needs to be clarified.

\section{Proof of Proposition~\ref{prop:var}}
This proposition really is just a variation on the work of Higman and Schupp. For~\ref{pt:var:no}, we start by recalling the following.

\begin{lem}[Higman's circular argument]\label{lem:Hig}
Let $f$ be a homomorphism from $\Hig_n$ to another group. If $f(a_i)$ has finite order for some $i$, then $f$ is trivial.
\end{lem}

\begin{proof}[Proof {\upshape (see also~\cite[p.~547]{Neumann54}).}]
The relations in~\eqref{eq:Hig} imply inductively that $f(a_i)$ has finite order $r_i \geq 1$ for \emph{all} $i$. Suppose for a contradiction that $r_i>1$ for some $i$, hence for all $i$ by the relations~\eqref{eq:Hig}. Let $p$ be the smallest prime dividing any $r_j$. The relation $a_{j-1}^{r_{j-1}} a_j a_{j-1}^{-r_{j-1}} = a_j^{2^{r_{j-1}}}$ implies that $2^{r_{j-1}}-1$ is a multiple of $r_j$ and hence of $p$. In particular, $p\neq 2$ and the order $s>1$ of $2$ in $(\ZZ/p\ZZ)^\times$ divides $r_{j-1}$. This contradicts the choice of $p$ because $s\leq p-1$.
\end{proof}

Suppose now that $f$ is a homomorphism from $K^{(n,x)}$ to an amenable group. The image of $K_i$ in $K^{(n,x)}$ is mapped by $f$ to a finite group, so that in particular $f(x_i)$ has finite order for all $i$. Since we have a homomorphism $\Hig_n\to K^{(n,x)}$ sending $a_i$ to $x_i$, we deduce from Lemma~\ref{lem:Hig} that  $f(x_i)$ is in fact trivial. Since $K$ is normally generated by $x$, it follows that $f(K_i)$ is trivial. We conclude that $f$ is trivial because the various $K_i$ generate $K^{(n,x)}$.

\bigskip
The two other points follow once we re-construct $K^{(n,x)}$ as a suitable amalgam. Recall that $x$ has infinite order; thus
\begin{equation}\label{eq:KL}
L=\langle K, h : [h,x]=x \rangle
\end{equation}
is an HNN-extension; we define $L_i, h_i$ similarly. Now
$$H = \big\langle L_0, L_1 : x_0 = h_1 \big\rangle$$
is a free product with amalgamation (because $x_0$ has infinite order) and therefore, using also the HNN-structure of~\eqref{eq:KL}, it follows that $\langle h_0, x_1 \rangle$ is a free group on $h_0, x_1$. Likewise, since $n\geq 4$, we deduce that
$$H' = \big\langle L_2, \ldots L_{n-1} : x_2 = h_3, \ldots, x_{n-2} = h_{n-1} \big\rangle$$
is a (successive) free product with amalgamation and that $h_2, x_{n-1}$ are a basis of a free group in $H'$. Therefore, we obtain $K^{(n,x)}$ by amalgamating $H$ and $H'$ over the groups $\langle h_0, x_1 \rangle$ and $\langle x_{n-1}, h_2 \rangle$ by identifying the free generators in the order given here.

Now the finiteness properties of~\ref{pt:var:fin} all follow since $K^{(n,x)}$ was obtained from copies of $K$ by finitely many HNN-extensions and amalgamated free products (see e.g.~\cite[\S7]{Geoghegan}). Regarding SQ-universality, P.~Schupp proved that it suffices to find a \emph{blocking pair} for $\langle h_0, x_1 \rangle$ in $H$, see Thm.~II in~\cite{Schupp71}. A blocking pair is provided for instance by any distinct non-trivial powers of an element $t\in H$ such that $t,  h_0, x_1$ form a basis of a free group (see the comment after Thm.~II in~\cite{Schupp71}). Just as in Lemma 4.3 of~\cite{Schupp71}, the element $t=x_0\inv x_1 h_0 x_1\inv x_0$ will do.

\section{Proof of Theorem~\ref{thm:G}}
We now turn to the groups $L$ and $G_n$ defined in part~1.B of the Introduction and fix some more notation. Denote by $\Heis(\alpha, \beta, \zeta)$ the (discrete) Heisenberg group with generators $\alpha,\beta$ and central generator $\zeta$. More precisely, it is defined by the relations $[\alpha, \beta]= \zeta$ and $[\zeta,\alpha]= [\zeta,\beta]=e$. For instance, $\{v,x,y\}$ (or just $\{v,x\}$) generate a copy of $\Heis(v,x,y)$ in $L$.

We shall use repeatedly, but tacitly, the following fundamental property of a free product with amalgamation $A*_C B$. If $A'<A$ and $B'<B$ are subgroups whose intersections with $C$ yield the same subgroup $C'<C$, then the canonical map $A' *_{C'} B' \to A*_C B$ is an embedding~\cite[8.11]{Neumann48}.

\medskip
We embed $L$ into a larger group $J$ generated by $L$ together with an additional generator $z$ by defining the following free product with amalgamation:
\begin{equation}\label{eq:J}
J = L *_{\langle h,v\rangle} \Heis(h,z,v).
\end{equation}
Although $h$ and $v$ already occur in our definition of $L$, there is no ambiguity since they form a basis of a copy of $\ZZ^2$ both in $L$ and in $\Heis(h,z,v)$. In particular, $L$ is indeed canonically embedded in $J$.

When we want to consider normal forms for this amalgamation (cf.~\cite[\S1]{Schreier27} or Thm.~4.4 in~\cite{Magnus-Karrass-Solitar}), it is convenient that there are very nice coset representatives of $\langle h,v\rangle$ in each factor. Indeed, in $\Heis(h,z,v)$, we can simply take the group $\langle z\rangle$. In $L$ written as
$$L =  (\ZZ[1/2])^2 \rtimes (\langle h \rangle \times \langle u,v \rangle ),$$
we can take as set of representatives the group $(\ZZ[1/2])^2\rtimes K$, where $K\lhd\langle u,v \rangle$ is the kernel of the morphism killing $v$.

\medskip
As before, we shall denote by $J_i$ a family of independent copies of $J$. We further denote by $z_i$ the corresponding additional generator. Then we have an equivalent presentation of $G_n$ given by
%
%
\begin{equation}\label{eq:newP}
\Bigg\langle J_i: 
\begin{array}{c}
v_{i-1} = h_{i}\\
x_{i-1} = z_{i}\\
y_{i-1} = v_{i}\\
z_{i-1} = u_{i}
\end{array}
\kern4mm\forall\, i\in \ZZ/n\ZZ \Bigg\rangle.
\end{equation}
The advantage is that each relation involves only \emph{successive} indices $i-1$ and $i$.

\medskip
We define inductively the groups $D_r$ for $r\in \NN$, starting with $D_0 = J_0$, by the presentation
%
%
$$D_r= \Bigg\langle D_{r-1} , J_r :
\begin{array}{c}
v_{r-1} = h_{r}\\
x_{r-1} = z_{r}\\
y_{r-1} = v_{r}\\
z_{r-1} = u_{r}
\end{array}
\Bigg\rangle.$$
We claim that this is in fact a free product with amalgamation of $D_{r-1}$ and $J_r$. More precisely, we claim that the subgroups of $J$ given respectively by
\begin{equation}\label{eq:QT}
Q=\langle v,x,y,z \rangle_J \ \text{ and } \ T=\langle h,z,v,u \rangle_J
\end{equation}
are isomorphic under matching their generators in the order listed in~\eqref{eq:QT}. This claim, transported to the various $J_i$, implies in particular by induction that $D_r$ is indeed a free product with amalgamation $D_r\cong D_{r-1} *_{Q_{r-1} = T_r} J_r$, where $Q_i, T_i$ denote the corresponding subgroups of $J_i$.

To prove the claim, we note first that the structure of $Q$ is revealed by observing which subgroups are generated by $\{v,x,y\}$ and by $\{v,z\}$ in the amalgamation~\eqref{eq:J} defining $J$. Both intersect $\langle h,v\rangle$ exactly in $\langle v\rangle$ and thus $Q$ is itself a free product with amalgamation $Q= \Heis(v,x,y) *_{\langle v\rangle} \langle v,z\rangle_J$ with $\langle v,z\rangle_J\cong \ZZ^2$.

As for $T$, given its relations, we have an epimorphism $Q\to T$ given by the above matching of generators; we need to show that it is in fact injective. To this end, consider that $T$ is generated by its subgroups $\Heis(h,z,v)$ and $\langle h,v,u\rangle_J$. Since $L$ is a factor of $J$, the latter is $\langle h,v,u\rangle_L \cong \ZZ\times F_2$. Thus $T$ is an amalgamated free product $\Heis(h,z,v) *_{\langle h,v \rangle} \langle h,v,u\rangle_L$. The injectivity now follows. In conclusion, $D_r$ is the following iterated free product with amalgamations:
$$D_r \cong J_0 *_{Q_0 = T_1}  J_1 *_{Q_1 = T_2} \cdots *_{Q_{r-1} = T_r}  J_r.$$
We also need to understand the intersection $Q\cap T$, which contains at least the group $\langle z,v \rangle_J \cong \ZZ^2$. In fact, this intersection is exactly $\langle z,v \rangle_J$. This follows by examining the normal form for the particularly simple choice of coset representatives made above.

As a consequence, we deduce that when $r\geq 3$, the subgroups $T_0$ and $Q_r$ of
\begin{equation}\label{eq:Dr2}
D_r \cong \big( J_0 *_{Q_0 = T_1}  J_1\big) *_{Q_1 = T_2}  \cdots  *_{Q_{r-2} = T_{r-1}}   \big( J_{r-1} *_{Q_{r-1} = T_r} J_r\big)
\end{equation}
intersect trivially and hence generate a free product  $T_0 * Q_r$. 

\medskip
Finally, to close the circle, we will use the assumption $n\geq 8$ and glue $D_{n-5}$ with a copy $D'_3$ of $D_3$ as follows. We shift indices in the $D_3$ factor to obtain the isomorphic group
$$D' _3 =  J_{n-4} *_{Q_{n-4} = T_{n-3}} \cdots *_{Q_{n-2} = T_{n-1}}  J_{n-1}.$$
In $D'_3$, the subgroups $T_{n-4}$ and $Q_{n-1}$ generate $T_{n-4}*Q_{n-1}$. Since we have constructed isomorphisms $T_0\cong Q_{n-1}$ and $Q_{n-5}\cong T_{n-4}$, we have a corresponding isomorphism
$$\fhi\colon T_0 * Q_{n-5} \lra  Q_{n-1} * T_{n-4}$$
and therefore we have a free product with amalgamation
\begin{equation}\label{eq:close}
D_{n-5} *_{\fhi} D'_3.
\end{equation}
Since this is a rewriting of the presentation~\eqref{eq:newP}, we have indeed constructed $G_n$ as an amalgam whenever $n\geq 8$. In particular, $L_i$ is embedded in $G_n$.

\medskip
At this point, we have established point~\ref{pt:G:Z} of Theorem~\ref{thm:G}, observing that $(\ZZ[1/2])^2 \rtimes \langle h \rangle$ is an HNN-extension of $\ZZ\times \ZZ$, that we can write
$$L \cong \Big( (\ZZ[1/2])^2 \rtimes \langle h \rangle\Big) \rtimes (\ZZ * \ZZ)$$
and that Heisenberg groups have the form $(\ZZ\times \ZZ) \rtimes \ZZ$.

On the other hand, point~\ref{pt:G:fin} follows from~\ref{pt:G:Z}, see e.g.~\cite[\S7]{Geoghegan}. As for~\ref{pt:G:noT}, we only need to recall that Kazhdan groups have Serre's property~FA~\cite[\S6.a]{Harpe-Valette_s}. This implies that any Kazhdan subgroup of $G_n$ can be recursively constrained into the factors of any amalgam. By~\ref{pt:G:Z}, we finally reach $\ZZ$, which has no non-trivial Kazhdan subgroup.

For~\ref{pt:G:SQ}, we indulge in the expedience of $n\geq 9$. This allows us to see from the decomposition~\eqref{eq:Dr2} applied to $r=n-5\geq 4$ that we have a free product
$$\langle T_0, u_2 x_2,  Q_r\rangle_{D_r} = T_0 * \langle u_2 x_2 \rangle *Q_r.$$
Indeed, reasoning within $J$, we see that $\langle u x \rangle$ intersects both $Q$ and $T$ trivially (and is infinite). This implies that any two distinct non-trivial powers of $u_2 x_2$ constitute a blocking pair for $T_0 *Q_r$ in $D_r$, see again~\cite{Schupp71}; we conclude that $G_n$ is SQ-universal.

\medskip
Turning to~\ref{pt:G:no}, we first observe that every generator in the presentation~\eqref{eq:Gn} functions as a self-destruct button for the group $G_n$, i.e.\ normally generates $G_n$.

\begin{lem}\label{lem:autothysis}
Let $f$ be a homomorphism from $G_n$ to another group. If $f$ sends some $x_i$ or some $y_i$ to the identity, then $f$ is trivial.
\end{lem}

\begin{proof}
The element $u_i v_i\inv u_i$ conjugates $x_i$ to $y_i\inv$ and therefore we can assume that $f(y_i)$ is trivial. Since $y_i=v_{i+1}$, the relation $[v_{i+1}, x_{i+1}] = y_{i+1}$ implies inductively that $f(y_j)$ vanishes for all $j$. Conjugating by $u_j v_j\inv u_j$, we find that all generators in~\eqref{eq:Gn} are trivialized by $f$.
\end{proof}

Let now $f$ be a homomorphism from $G_n$ to some Haagerup group. The subgroup $\langle x,y\rangle$ of $\langle x,y\rangle \rtimes \langle u,v\rangle$ has the relative property~(T). Indeed, the proof of the corresponding statement for $\ZZ^2\rtimes \SL_2(\ZZ)$ only depends on the image of $\mathbf{SL}_2(\ZZ)$ in the automorphism group of $\ZZ^2$, see e.g.~\cite{Burger_kazhdan}. Therefore, $f(\langle x_i, y_i\rangle)$ is finite for all $i$.

On the other hand, the presentation~\eqref{eq:G} shows that we have a morphism $\Hig_n \to G_n$ defined by $a_i \mapsto y_{2i}$. By Higman's argument (Lemma~\ref{lem:Hig}), it follows that $f(y_{2i})$ is trivial for all $i$. We conclude from Lemma~\ref{lem:autothysis} that $f$ is trivial.

\smallskip
For~\ref{pt:G:Folner}, we use the explicit \emph{relative Kazhdan pair} $(S_0, \epsilon_0)$ provided by M.~Burger, Example~2 p.~40 in~\cite{Burger_kazhdan}. Here $S_0$ is a certain generating set of $\ZZ^2\rtimes \SL_2(\ZZ)$ and $\epsilon_0 = \sqrt{2-\sqrt 3}$. Being a relative Kazhdan pair means that any unitary representation of $\ZZ^2\rtimes \SL_2(\ZZ)$ with $(S_0, \epsilon_0)$-invariant vectors admits $\ZZ^2$-invariant vectors, see~\cite{Harpe-Valette_s}. We denote by $S=\{x,y,u,v\}$ our usual generators of $\ZZ^2\rtimes F_2$ and write $\overline{S} = S \cup S\inv \cup \{e\}$; then $(S,\epsilon)$-invariance is equivalent to $(\overline{S},\epsilon)$-invariance. The set $S_0$ from~\cite[Ex.~2]{Burger_kazhdan} is contained in $\overline{S}^3$ under the map $F_2\to \SL_2(\ZZ)$ and therefore every $(S,\epsilon_0/3)$-invariant vector is $(S_0, \epsilon_0)$-invariant. Now~\ref{pt:G:Folner} follows because $\epsilon_0/3 > 1/6$ and because any $(S, \epsilon)$-F{\o}lner set gives a  $(S, \sqrt\epsilon)$-invariant vector.

\begin{rem*}
The corresponding argument provides also a lower bound on F{\o}lner constants for quotients of $K^{(n,x)}$ when $K$ is Kazhdan.
\end{rem*}

It only remains to prove~\ref{pt:G:trivial}. Consider again the homomorphism $\Hig_n \to G_n$ above. When $n$ is even, this factors through a morphism $\Hig_{n/2} \to G_n$. Since $\Hig_r$ is trivial for $r\leq 3$ (see~\cite{Higman51}), it follows that $y_0$ is trivial when $n=4,6$; now Lemma~\ref{lem:autothysis} shows that $G_n$ is trivial. The same argument applied to the original map $\Hig_n \to G_n$ takes care of $n\leq 3$.


\bibliographystyle{../BIB/amsplain}
\bibliography{../BIB/ma_bib}

\end{document}